\documentclass[11pt]{amsart}
\usepackage[margin=3cm]{geometry}
\usepackage{amssymb, amsmath,mathtools}
\mathtoolsset{showonlyrefs,showmanualtags}
\usepackage{comment}
\usepackage{hyperref}
\usepackage{ytableau}

\newcommand{\gr}{\operatorname{gr}}
\newcommand{\Der}{\operatorname{Der}}
\newcommand{\Hom}{\operatorname{Hom}}
\newcommand{\Tor}{\operatorname{Tor}}
\newcommand{\End}{\operatorname{End}}

\newcommand{\prol}{\operatorname{prol}}
\newcommand{\m}{\mathfrak{m}}
\newcommand{\g}{\mathfrak{g}}

\newcommand{\vg}{\mathfrak v}
\newcommand{\ve}{\mathfrak{ve}}
\newcommand{\vf}{\mathfrak{vf}}
\newcommand{\sll}{\mathfrak{sl}}
\newcommand{\h}{\mathfrak{h}}
\newcommand{\eg}{\mathfrak{e}}
\newcommand{\fg}{\mathfrak{f}}
\newcommand{\R}{\mathbb{R}}
\newcommand{\bbC}{\mathbb{C}}
\newcommand{\Z}{\mathbb{Z}}
\newcommand{\dd}[1]{\frac{\partial}{\partial{#1}}}

\newcommand{\cL}{\mathcal L}
\newcommand{\cT}{\mathcal T}
\newcommand{\E}{\mathcal E}

\usepackage{todonotes}

\numberwithin{equation}{section}

\title[Finite type distributions via abnormal extremals]{Generalized pseudo-product structures and finite type distributions via abnormal extremals}

\thanks{ I.\ Zelenko was partly supported by NSF grant DMS 2105528.}

\author{Boris Doubrov}
\address{Belarussian State University, 
Nezavisimosti Ave.~4, Minsk 220030, Belarus;
 E-mail: boris.doubrov@gmail.com}

\author{Igor Zelenko}
\address{Igor Zelenko\\
         Department of Mathematics\\
         Texas A\&M University\\
         College Station\\
         Texas \ 77843\\
         USA}
\email{zelenkotamu@tamu.edu}
\urladdr{https://people.tamu.edu/~zelenkotamu/}

\newtheorem{thm}{Theorem}[section]

\newtheorem{lem}[thm]{Lemma}
\newtheorem{cor}[thm]{Corollary}
\newtheorem{prop}[thm]{Proposition}
\newtheorem{rem}[thm]{Remark}

\newtheorem{df}[thm]{Definition}
\newtheorem{ex}[thm]{Example}

\begin{document}
\begin{abstract}
We generalize the classical Tanaka result on the finiteness of the Lie algebra of infinitesimal symmetries for non-degenerate pseudo-product structures to the case when the completely integrable distributions defining the pseudo-product structure are no longer concentrated in the degree $-1$.
In order to do this, we modify the notion of universal prolongation of graded nilpotent Lie algebras and generalize the original finiteness criterion of Tanaka. 
Using this result, we demonstrate that 
in real analytic category distributions that are controllable by concatenations of regular abnormal extremal trajectories, also known as singularly transitive, have finite-dimensional symmetries. This result settles Problem V in the affirmative from the 2013 list of open problems by Andrei Agrachev. Additionally, we discuss applications to symmetries and natural equivalence problems for systems of ODEs of mixed order.
\end{abstract}
\maketitle

\section{Introduction}
\subsection{Filtered manifolds and their symbol}

Following N.~Tanaka \cite{tan1970} and T.~Morimoto \cite{mori}, we say that a smooth manifold $M$ is \emph{filtered} if its tangent bundle is equipped with filtration:
\[
0 = T^0M \subset T^{-1}M \subset \dots \subset T^{-\mu}M=TM,
\]
compatible with a Lie bracket:
\[
[T^{-i}M, T^{-j}M] \subset T^{-i-j} M,\quad \text{for all }i,j=1,\dots,\mu. 
\]
Here we denote by $T^{-i}M$ the subbundle of the tangent bundle as well as its set of smooth sections, i.e., those vector fields $X$ on $M$ such that $X_p\subset T_p^{-i}M$ for all $p\in M$.

For each $p\in M$ denote by $\gr T_pM$ the graded vector space associated with the above filtration. It is naturally equipped with a structure of a graded Lie algebra:
\[
[X_p + T^{-i+1}M, Y_p + T^{-j+1}M] = [X,Y]_p + T^{-i-j+1}M,
\]  
for all $X\in T^{-i}M$, $Y\in T^{-j}M$. It is easy to see that this bracket depends only on the values of the vector fields $X$ and $Y$ at the point $p$. The graded Lie algebra $\gr T_pM$ is called a \emph{symbol} of the filtered manifold $M$ at a point $p$. If all Lie algebras $\gr T_pM$ are isomorphic to a fixed graded Lie algebra $\m$, we say that a filtered manifold $M$ \emph{has a constant symbol} or \emph{is of type $\m$}. Note that by definition, all Lie algebras $\gr T_pM$ are negatively graded and therefore nilpotent.

Any smooth manifold can be considered as a filtered manifold with trivial filtration $T^{-1}M=TM$. The symbol of such filtration at each point $p\in M$ is the whole tangent space $T_pM$ considered as a commutative Lie algebra with a trivial grading, when all elements have degree $-1$.

\begin{ex}\label{ex:d_filt}
 Let $D$ be an arbitrary bracket-generating distribution on $M$. We define the filtration on $M$ by the so-called weak derived flag of $D$:
\begin{align}
\label{week_flag}
T^{-1}M&= D;\\
 T^{-i-1}M &= T^{-i}M + [T^{-1}M,T^{-i}M],\quad i\ge 1.
\end{align}
 The symbol of such filtration is a very important invariant of the distribution $D$. However, it is not constant in general. 
\end{ex}

Note also that by definition, the Lie algebra $\gr T_pM$ in the above example is generated by elements of degree $-1$. This property plays a crucial role in the original Tanaka work, and he calls such graded nilpotent Lie algebras \emph{fundamental}. One of the goals of the present paper is to generalize Tanaka constructions to the case of non-fundamental  symbols. This is motivated by the following generalization of the above example.  

\begin{ex}\label{ex-filtered-D}
We say that a bracket-generating vector distribution $D$ on $M$ is filtered, if it is itself equipped with a filtration:
\[
0 = \mathcal \mathcal F^0D \subset \mathcal \mathcal F^{-1}D \subset \dots \subset \mathcal \mathcal F^{-\nu}D=D,  
\]  
compatible with a Lie bracket of vector fields:
\[
[\mathcal \mathcal F^{-i}D, \mathcal \mathcal F^{-j}D]\subset \mathcal \mathcal F^{-i-j}D,\quad\text{for all }i,j=1,\dots,\nu, i+j\le \nu.
\]
We can naturally extend this filtration to the filtration of the manifold $M$ as follows:
\begin{align*}
T^{-i}M &= \mathcal \mathcal F^{-i}D,\quad i=1,\dots,\nu;\\
T^{-i-1}M &= T^{-i}M + \sum_{r+s=i+1} [T^{-r}M, T^{-s}M],\quad i>\nu.
\end{align*}
The symbol of such filtration is no longer a fundamental graded nilpotent Lie algebra. However, it still has the property that $\gr T_pM$ is generated by elements of degree $-1,\dots,-\nu$. 
\end{ex}

The simplest example of filtrations with non-fundamental symbol is when each $T^{-i}M$ is a completely integrable distribution. Then $\gr T_pM$ is a commutative Lie algebra for each $p\in M$. However, it may still have a non-trivial grading. 

More examples of filtered manifolds with non-fundamental symbol will appear in the generalization of pseudo-product structures in Section~\ref{s:gpps} and its application to the controllability of non-holonomic distributions via abnormal extremals in Section~\ref{abnormal_sec}.

\subsection{Universal Tanaka prolongation}
Let $\m$ be a (negatively) graded nilpotent Lie algebra. Universal Tanaka prolongation is defined as the largest graded Lie algebra $\g(\m)$ that satisfies the following conditions:
\begin{enumerate}
    \item $\g_i(\m)=\m_i$ for all $i<0$;
    \item for any $u\in\g_i(\m)$, $i\ge 0$, the equality $[u,\m]=0$ implies $u=0$.
\end{enumerate}

Let $\Der_0(\m)$ be the Lie algebra of all degree preserving derivations of $\m$ and let $\g_0$ be its arbitrary subalgebra. Then Tanaka also defines $\g(\m,\g_0)$ by additionally assuming that $\g_0(\m,\g_0)=\g_0$ in the above definition.  

While the original Tanaka papers define this universal prolongation only for fundamental graded nilpotent Lie algebras, it is easy to see that the graded Lie algebras $\g(\m)$ and $\g(\m,\g_0)$ are well-defined without this assumption.

\subsection{Pseudo-product structures}
The notion of a pseudo-product structure was introduced by N.~Tanaka~\cite{tan1970} and generalizes the notion of an almost product structure in the classical differential geometry.

\begin{df}\label{df:pp_struct} 
A \emph{pseudo-product} structure on a smooth manifold $M$ is a pair $(E,F)$ of completely integrable distributions such that $E\cap F=0$ and $D=E+F$ is a bracket-generating distribution. 
\end{df}

Pseudo-product structures naturally appear in the geometric interpretation of ordinary differential equations.

Let $(E,F)$ be an arbitrary pseudo-product structure on a smooth manifold $M$. Let $D=E+F$ be the corresponding bracket- generating distribution. According to Example~\ref{ex:d_filt} it defines the filtration $\{T^{-i}M\}$ on $M$ such that $T^{-1}M=D$. Moreover, the pair $(E,F)$ can be viewed as a decomposition of $T^{-1}M$ into the sum of two subspaces: $T^{-1}M=E\oplus F$. 

The integrability of distributions $E$ and $F$ implies that the corresponding subspaces $E_p$ and $F_p$ in the symbol algebra $\gr T_pM$ are, in fact, commutative subalgebras. We shall call the triple $(\gr T_pM; E_p, F_p)$ the symbol of the pseudo-product structure at the point $p\in M$. 

Conversely, let $\m$ be an arbitrary graded fundamental Lie algebra (i.e., $\m$ is generated by $\m_{-1}$) and let $\m_{-1}=\eg\oplus \fg$ be the decomposition of $\m_{-1}$ into the sum of two commutative subalgebras. Then we can construct a \emph{standard} pseudo-product structure with the symbol $(\m; \eg, \fg)$ as follows. Let $M$ be a simply connected Lie group with a Lie algebra $\m$. Then we extend the subspaces $\eg$, $\eg$ from $T_eM = \m$ to the whole manifold $M$ by left shifts. It is easy to see that the resulting distributions $E$ and $F$ will be completely integrable and will define a pseudo-product structure of type $(\m;\eg,\fg)$ in $M$. 

For each point $p\in M$ we have a so-called \emph{Levi bracket}
\[
L\colon \wedge^2  T_p^{-1}M \to T_p^{-2}M/T_p^{-1}M,
\]
which is a part of the multiplication in the symbol $\gr T_pM$. It is clear that $L(E_p,E_p)=L(F_p,F_p)=0$. So, $L$ is completely determined by its restriction to $E_p\times F_p$. We shall denote the corresponding bilinear map:
\[
E_p \times F_p \to T^{-2}M/T^{-1}M
\]
by the same letter $L_p$.

\begin{df} The pseudo-product structure $(E,F)$ on a smooth manifold $M$ is called \emph{non-degenerate}, if the bilinear map $L\colon E_p\times F_p\to T^{-2}M/T^{-1}M$ has trivial left and right kernels for each $p\in M$.   
\end{df}
Equivalently, non-degeneracy of the pseudo-product structure $(E,F)$ means that if $X\in E$ ($Y\in F$) and $[X,F]\subset E+F$ ($[Y,E]\subset E+F]$), then $X=0$ (resp. $Y=0$).

The following results on the pseudo-product structures were proved in the papers of Tanaka~\cite{tan1970} and Yatsui~\cite{yatsui} and are presented here in a slightly modified form for completeness.
\begin{thm}[\cite{tan1970}]
Let $(E,F)$ be a non-degenerate pseudo-product structure on a smooth manifold $M$ of constant type $(\m;\eg,\fg)$. Let $\g_0$ be a subalgebra in $\Der_0(\m)$ preserving the decomposition $\g_{-1}=\eg\oplus\fg$.
\begin{enumerate}
\item The Tanaka prolongation $\g(\m, \g_0)$ is finite-dimensional.
\item The symmetry algebra of the pseudo-product structure $(E,F)$ is finite-dimensional. Moreover, its dimension does not exceed $\dim \g(\m,\g_0)$, and equality is achieved if and only if $(E,F)$ is locally equivalent to the standard pseudo-product structure of type $(\m,\g_0)$. 
\end{enumerate}
\end{thm}

Let us outline its proof. It is based on several ideas. First, recall the classical notion of standard prolongation~\cite{stern64} of a subspace $A\in V^*\otimes W$. The \emph{$k$-th standard (Spencer--Sternberg) prolongation $A_k\subset S^{k+1}(V^*)\otimes W$} is defined inductively as:
\begin{align*}
A_0 &= A;\\
A_k &= \big(S^{k+1}(V^*)\otimes W\big) \cap \big(V^*\otimes A_{k-1}\big),\quad k\ge 1.
\end{align*}
Finally, set
\[
\prol(A) = \bigoplus_{i\ge 0} A_i.
\]

We say that $A$ is of finite type if $\prol(A)$ is finite-dimensional. This is equivalent to $A_k=0$ for sufficiently large $k$. It is well known~\cite{spencer} that $A$ is of finite type if and only if its complexification $A^{\bbC}$ contains no element of rank~$1$.

Next, Tanaka~\cite{tan1970} reduces the question of the finite dimensionality of $\g(\m,\g_0)$ for arbitrary subalgebra $\g_0\subset \Der(\m,\g_0)$ to the standard prolongation. That is, he defines a graded subalgebra $\h$ in $\g(\m,\g_0)$ as follows:
\begin{align*}
\h_{-i}&=0,\quad i\ge 2;\\
\h_{-1}&=\m_{-1};\\
\h_{j}&=\{u\in \g(\m,\g_0)\mid [u,\m_{-i}]=0\text{ for }i\ge 2\}.
\end{align*}
Then it is easy to see that $\h_k$ belongs to $k$-th prolongation of $\h_0$ for all $k\ge 1$, where $\h_0$ is viewed as a subspace in $\Hom(\m_{-1},\m_{-1})=\m_{-1}^*\otimes \m_{-1}$. 

The Tanaka finite-dimensionality criterion states that $\g(\m,\g_0)$ is finite-dimensional if and only if $\h$ is finite-dimensional. Using rank 1 criterion, we find that $\g(\m,\g_0)$ is finite-dimensional if and only if $\h_0^{\bbC}$ contains no rank~1 elements.

Finally, this criterion is applied to the subalgebra $\g_0\subset \Der_0(\m)$ consisting of all elements that preserve the decomposition $\m_{-1}=\eg\oplus \fg$. In this case, it is sufficient to note that any derivation in $\Der_0(\m)$ that preserves both $\eg$ and $\fg$ and vanishes on $\m_{-2}=[\eg,\fg]$ cannot be of rank 1 due to the non-degeneracy of the pseudo-product structure.

Even though the above theorem is already very strong and has a number of very important applications, it cannot be applied directly to the geometry of arbitrary non-holonomic distributions and to the related pseudo-product structures for the following reasons:
\begin{enumerate}
\item Generic non-holonomic distributions very rarely have a constant
symbol. And Tanaka theory cannot be applied directly to the case
when symbols $\gr T_p$ are not isomorphic at different points.
\item The condition that the symbol algebra $\m$ is generated by $\m_{-1}$ is crucial in the Tanaka criterion in the proof of the fact that $\dim \g(\m,\g_0)$ is finite-dimensional. However, this is not always the case, and we need to consider more general filtered manifolds with symbol, not necessarily generated  by $\m_{-1}$.
\item The above results cannot be applied to the degenerate pseudo-product structures. And this type of structure does appear when one uses abnormal extremals for the construction of absolute parallelism for non-holonomic distributions and the study of their symmetries.
\end{enumerate}   

The main purpose of the current paper is to address these limitations by generalizing the notion of pseudo-product structures and the Tanaka finiteness criterion. This allows us to apply these generalizations to control theory, specifically linking the controllability of non-holonomic distributions by abnormal extremal trajectories, referred to as singular transitivity in the terminology of \cite{agrachev}, to the finite dimensionality of their symmetry algebras (see Corollaries \ref{cor_control} and \ref{cor_min_rank} below). This result provides a positive answer to Problem V from Agrachev’s 2013 list of open problems \cite{agrachev}.

In the last Section we discuss applications to the geometry of systems of ODEs of mixed order. The case of two ODEs was treated in~\cite{mixed-order}. In this paper we extend these results to an arbitrary number of equations.

\section{Generalized pseudo-product structures}\label{s:gpps}
\subsection{Generalization of the universal Tanaka prolongation}
We shall use the following generalization of the universal Tanaka prolongation. Let $\eg_1, \dots, \eg_k$ be an arbitrary set of graded subalgebras in $\m$. Define \emph{the generalized Tanaka prolongation of $(\m; \eg_1, \dots, \eg_k)$} as the largest graded Lie algebra $\g=\g(\m; \eg_1, \dots, \eg_k)$ that satisfies the above two conditions for the universal Tanaka prolongation
\begin{enumerate}
    \item $\g_i(\m)=\m_i$ for all $i<0$;
    \item for any $u\in\g_i(\m)$, $i\ge 0$, the equality $[u,\m]=0$ implies $u=0$.
\end{enumerate}
and the additional condition
\begin{enumerate}
    \item[(3)] for any $u\in \g_i$, $i\ge 0$, and any subalgebra $\eg_l$, $1\le l\le k$ we have
    \[
    [u,\eg_l]\subset \eg_l + \sum_{j\ge 0}\g_j.
    \]
\end{enumerate}
In particular, the first two conditions imply that $\g$ is a subalgebra of $\g(\m)$.

An important example of this definition is the case where all subalgebras $\eg_i$ are concentrated in degree $-1$. Then they are automatically abelian, and $\g(\m; \eg_1, \dots, \eg_k)$ coincides with the original universal Tanaka prolongation $\g(\m,\g_0)$, where $\g_0$ is a subalgebra in $\Der_0(\m)$ that preserves all subspaces $\eg_i\subset \g_{-1}$.

\begin{ex}
    Let $\m$ be a 3-dimensional Heisenberg Lie algebra with basis elements $\{e_1,e_2,e_3\}$ of degrees $-1,-1,-2$ respectively and the only non-trivial bracket $[e_1,e_2]=e_3$. 

    Let $\eg=\langle e_1\rangle$ and $\fg=\langle e_2\rangle$ be two subalgebras concentrated in degree $-1$. It is well-known that the Tanaka prolongation $\g(\m;\eg,\fg)$ is isomorphic to $\sll(3)$ with $\m$ being its subalgebra of strictly upper-triangular matrices. 

    Now choose a different $\fg=\langle e_2, e_3\rangle$, which is no longer concentrated in degree $-1$. The generalized Tanaka prolongation in this case becomes a 6-dimensional subalgebra $\mathfrak{gl}(2)\rightthreetimes \R^2\subset \sll(3)$ that corresponds to the subgroup of affine transformations inside the group of projective transformations of the plane.
    
    This is the simplest example of the case where the generalized Tanaka prolongation differs from the classical definition.
\end{ex}

\subsection{Notion of generalized pseudo-product structures}
We generalize Definition~\ref{df:pp_struct} as follows. 

\begin{df} A \emph{generalized pseudo-product structure} on a filtered manifold $M$ is given by a pair of completely integrable distributions $(E,F)$ of $T^{-\nu}M$ such that the following conditions are satisfied:
\begin{enumerate}
\item $E\cap F = 0$;
\item $E\cap T^{-i}M + F\cap T^{-i}M = (E+F)\cap T^{-i}M$ for all $i>0$;
\item $E+F$ is bracket-generating.
\end{enumerate}
\end{df}

In applications, we usually assume that $E\oplus E = T^{-\nu}M$ for some $\nu>0$. If $\nu=1$, then in case we have $E\oplus F=T^{-1}M$, and the above definition reduces to Definition~\ref{df:pp_struct}. However, in general, $\m$ is no longer generated by $\m_{-1}$, but only by $\oplus_{i=1}^\nu \m_{-i}$.  

For each point $p\in M$ the symbol of a pseudo-product structure is defined as a triple $(\m;\eg,\fg)$, where $\m=\gr T_pM$ is a symbol of a filtered manifold $M$ at $p$, $\eg$ and $\fg$ are two graded subspaces in $\m$ given by:
\begin{align*}
\eg_{-i} &= E_p \cap T^{-i}M \mod T^{-i+1}M \subset \m_{-i};\\
\fg_{-i} &= F_p \cap T^{-i}M \mod T^{-i+1}M \subset \m_{-i},
\end{align*}
for all $i>0$.

As before, integrability of $E$ and $F$ implies that both $\eg$ and $\fg$ are graded subalgebras in $\m$. But they are not necessarily abelian, have trivial intersection, and their sum generates $\m$.

\begin{df}
\label{nondeg_def}
The pseudo-product structure is \emph{non-degenerate} if its symbol $(\m;\eg,\fg)$ satisfies the following condition: the bilinear map
\begin{equation}
\label{Leviform}
\eg \times \fg \to \m,\quad (x,y)\mapsto [x,y] \quad\text{for any }x\in\eg,y\in\fg,
\end{equation}
has trivial left and right kernels. 
\end{df}

As in the case of $E\oplus F=T^{-1}M$, this condition is equivalent to the following. If $X\in E$ ($Y\in F$) and $[X,F]\subset E+F$ ($[Y,E]\subset E+F$), then $X=0$ (resp. $Y=0$).

    The notion of pseudo-product structures turns out to be very useful in the study of \emph{degenerate} pseudo-product structures in the classical sense. Namely, let $E$ and $F$ be two completely integrable distributions on a smooth manifold $M$ such that $E\cap F=0$, $E$ is a filtered distribution
    $$E^{-1}\subset E^{-2}\subset \cdots \subset E^{-\mu}=E$$ such that 
    \begin{equation}
    \label{E_brackets_exact}
    E^{-i-1}=(E^{-1})^{i+1}=[E^{-1}, E^{-i}]+E^{-i}, \quad  i=1, \ldots \mu-1,
    \end{equation}
     and
    $D=E+F$ is bracket-generating. Obviously, by \eqref{E_brackets_exact}, bracket-genericity of $D$ is equivalent to braket genericity of $F+E^{-1}$.
    
 Let $F^{-1}:=F$. Refine further filtrations on $E$ and $F$ introducing:
    \begin{align}
        ~& E^{i}=\{ X\in E^{i-1} \mid [X, F]\subset E^{i-1}+F\}, \quad i\ge 0, \label{E_filt_down}\\
         ~& F^{i}=\{ Y\in F^{i-1} \mid [Y, E^{-1}]\subset  F^{i-1}+E^{-1}\}, \quad i\ge 0.\label{F_filt_down}   
    \end{align}
    Conditions \eqref{E_brackets_exact} together with the Jacobi identities imply that $[F^i, E^j]\subset F^{i+j}+E^j$ for any $j<0$ and any $i$.
    Restricting to an open subset, if needed, we can assume that for an open set $U \subset M$ and for every $i \geq 0$, $\dim  E^i_p$ and $\dim  F^i_p$ are constant for every $p \in M$, $i\geq 0$. 

\begin{rem}
\label{Int_rem}
Note that  by \eqref{E_filt_down}-\eqref{F_filt_down} integrability of $F$  implies that $F^i$ are integrable for any~$i\ge 0$.     
\end{rem}
\begin{df}
\label{Freeman}
    By analogy with terminology in the CR-geometry we will call filtrations \eqref{E_filt_down}-\eqref{F_filt_down} the \emph{Freeman filtrations} generated by the pair $(E, F)$. 
\end{df}

    \begin{ex}
\label{osculation}
    Suppose now that $E^0=0$ and $F^{\nu-1}\ne 0$, $F^{\nu}=0$ for some $\nu\ge0$. 
    Define a filtration in $D$ as follows:
    \begin{equation}
    \label{extended_filt}
        \mathcal \mathcal F^{-i}D = E^{-i}+F^{\nu-i},\quad i=1,\dots,\nu.
    \end{equation}
    Extend this filtration to the filtration on $M$ as described in Example~\ref{ex-filtered-D}. Then the pair $(E,F)$ defines a non-degenerate generalized pseudo-product structure on $M$ with the constructed filtration. 
    
    Note that if $\nu\ge 1$ (and $\mu=1$), then the pseudo-product structure $(E, F)$ is degenerate in the classical sense.
\end{ex}

\subsection{Generalization of Tanaka criterion to non-fundamental Lie algebras}
In this subsection we generalize Tanaka finiteness criterion to the case of non-fundamental graded nilpotent Lie algebra $\m$. 

Let $\m=\sum_{i=-\mu}^{-1}\m_i$ be an arbitrary finite-dimensional negatively graded nilpotent Lie algebra. Unlike Tanaka, we do not assume that $\m$ is fundamental (i.e., generated by elements of degree $-1$). 

Let $E=\sum_{i\in\Z} E_i$ be an arbitrary graded $\m$-module. All graded $\m$-modules considered in this subsection will satisfy the following conditions:
\begin{enumerate}
\item $\dim E_i < \infty$ for all $i\in\Z$;
\item $E$ is \emph{negatively bounded}, that is the negative part of $E$ is finite-dimensional. This is equivalent to $E_i=0$ for sufficiently small~$i$.  
\end{enumerate}

Following Tanaka, we say that the $\m$-module $E$ satisfies condition~(C), if the equality $\m x=0$ implies $x=0$ for all $x\in E_i$, $i\ge0$.  The examples of such modules are given by Tanaka prolongations $\g(\m,\g_0)$ for any $\g_0\subset \Der_0(\m)$.

Another important example of modules satisfying condition~(C) is the standard prolongation as exposed in \cite{gsst1967,spencer}. In this theory $\m=V$ is a finite dimensional vector space, which can be viewed as a graded commutative Lie algebra concentrated in degree $-1$. Due to condition~(C) $E_0$ can be identified with a subspace in $\Hom(V, E_{-1})$, and $E_i$ for $i>0$ can be viewed as subspaces in the standard (Sternberg) prolongation $E^{(i)}_0\subset \Hom(S^{i+1}V,E_{-1})$ of $E_0$. 

Let $\h=[\m,\m]$ be the derived sublgebra of $\m$. It is clear that $\h$ is also graded. Let $\nu$ be the largest integer $i$ such that $\h_{-i}\ne 0$. So, we have $\h=\displaystyle{\sum_{i=-\nu}^{-1}\h_i}$. Denote by $H(\m)$ the quotient $\m/\h$, which is a graded commutative Lie algebra. If $\m$ is fundamental, then $\h=\oplus_{i\ge 2}\m_{-i}$ and $\m/\h$ is isomorphic to $\m_{-1}$ as a vector space. However, in the general case of non-fundamental $\m$ the quotient algebra $H(\m)$ may have elements of lower degree. 

Let $H(E)$ be the subspace in $E$ consisting of all elements annihilated by~$\h$. Then $H(E)$ is a submodule of $E$. Moreover, we can naturally equip it with the structure of $H(\m)$-module. It is clear that $H(\m)$-module $H(E)$ does also satisfy condition~(C).

The theorem below is a generalization of the Tanaka finiteness criterion formulated in~\cite{tan1970}. We follow closely the preprint by I.~Naruki~\cite{naruki} for the proof. It is valid over an arbitrary base field $k$ of characteristic $0$ (which is either $\mathbb{R}$ or $\mathbb{C}$ in the scope of this paper).
\begin{thm}\label{E_H(E)_thm}
Let $E$ be a graded negatively bounded $\m$-module. Then $E$ is finite-dimensional if and only if the $H(\m)$-module $H(E)$ is also finite-dimensional.
\end{thm}  
\begin{proof}
If $E$ is finite-dimensional, then so is $H(E)$. So, we need only to prove the converse. Assume that $H(E)$ is finite-dimensional. The proof is done by induction by $\nu$ (the depth of the grading of $\h=[\m,\m]$), and by the dimension of $\h_{-\nu}$. The theorem is trivial for $\h=0$. So, we assume that $\nu\ge 2$ and $\h_{-\nu}\ne 0$. Note that $\h_{-\nu}$ lies in the center of $\m$.

 Let $u$ be an arbitrary non-zero element in $\h_{-\nu}$. Define $E'=\{a\in E\mid ua=0\}$. As an inductive step, assume $\dim E'<\infty$ and prove that this implies $\dim E<\infty$. 

Define the dual module $E^*=\oplus E_i^*$ with the natural action of $\m$. Note that it is also a graded module with $\deg E^*_i=-\deg E_i$. Note that $E^*$ is \emph{positively bounded}, that is $E^*_i=0$ for sufficiently large $i$.

Define the action of $k[X]$ on $E^*$ by assuming that $X\alpha=u\alpha$ for any $\alpha\in E^*$. 

\begin{lem}
The module $E^*$ is finitely generated over $k[X]$.     
\end{lem}
\begin{proof}
The condition $\dim E'<\infty$ is equivalent to $\dim E^*/uE^* < \infty$. Note that $uE^*$ is also graded. Choose an arbitrary graded subspace $S\subset E^*$ complementary to $uE^*$. Let $k[X]S$ be the $k[X]$-submodule of $E^*$ generated by $S$. As $k[X]S$ contains $S$, any element from $E^*/k[X]S$ has a representative of the form $u\alpha$ for some $\alpha \in E^*$. Hence, we have 
\begin{equation}\label{eq:u-act}
E^*/k[X]S=u(E^*/k[X]S).
\end{equation}
As the module $E^*/k[X]S$ is also graded and positively bounded and $u$ has strictly negative degree, the equality~\eqref{eq:u-act} is possible only if $E^*/k[X]S=0$. Since $S$ is finite-dimensional, this proves that $E^*=k[X]S$ is finitely generated over $k[X]$.
\end{proof}

To complete the proof, we use the structure theory of finitely generated modules over principal ideal domains. It states that $k[X]$-module $E^*$ is isomorphic to
\[
k[X]^r + \sum_{i=1}^s k[X]/(q_i)
\]
for some $r\ge 0$, $s\ge 0$ and a sequence of non-zero polynomials $q_i\in k[X]$. The second summand is the same as \emph{the torsion of $E^*$}. It can be defined as
\[
tE^* = \{\alpha \in E^* \mid p\alpha=0\text{ for some }0\ne p\in k[X]\}.
\]
It is clear that $tE^*$ is a submodule of the $k[X]$-module $E^*$. Moreover, as $u$ lies in the center of $\m$, the action of $k[X]$ commutes with the action of $\m$ on $E^*$. So, the quotient $E^*/tE^*= k[X]^r$ is equipped with both actions of $k[X]$ and $\m$.

Since $u$ also lies in $[\m,\m]$, it can be represented as 
\[
u = \sum_i [a_i,b_i],\quad a_i,b_i\in \m.
\]
The action of each element $a_i$, $b_i$ on $k[X]^r$ is given by a certain $r$ by $r$ matrix $A_i$, $B_i$ over $k[X]$. Hence, the action of $[a_i,b_i]$ is given by the matrix $[A_i,B_i]$, which has trace $0$. Thus, we get that the action of $u$ is also represented by the matrix with trace $0$. On the other hand, $u$ acts as a multiplication by $X$, so it has trace $rX$. This implies that $r=0$ and $E^* = tE^*$. Since each of the $k[X]$-modules $k[X]/(q)$, $q\ne 0$, is a finite-dimensional vector space over $k$, we see that $\dim E^*<\infty$. 
\end{proof}
\begin{cor}
    Let $\g(\m)$ be the universal Tanaka prolongation of $\m$ and let $\g$ be an arbitrary graded subalgebra in $\g(\m)$ such that $\g_{-}=\m$. Then $\g$ is finite-dimensional if and only if the centralizer $Z_{\g}([\m,\m])$ is finite-dimensional.
\end{cor}
\begin{proof}
    Apply the theorem to the case $E=\g$.  
\end{proof}

Let $\g$ be a graded subalgebra in $\g(\m)$ with $\g_{-}=\m$. Set $V=\m/[\m,\m]$ be an abelian Lie algebra treated as a vector space.  Define $\g^0=\displaystyle{\bigoplus_{i\geq 0}\g_i}$. 

As in the Corollary, set $E=\g$ and $H(\g)=Z_{\g}([\m,\m])$. Then the adjoint action of $H(\g)^0=H(\g)\cap \g^0$ on $\g$ modulo $\g^0$ induces a linear subspace $A(\g)\subset \End(V)$. 
In other words:
\begin{equation}
\label{A_def}
A(\g) = \left\{ \bar u \colon V\to V:  \begin{array} {l}\exists\, u\in H(\g)^0 \text{ such that }  \forall\, x\in V\\ \bar u(x) = [u,x] \mod \g^0\end{array}\right\}.
\end{equation}
In fact $A(\g)$ is a subalgebra in $\mathfrak{gl}(V)$, but we will not use this in the proof.
\begin{thm}\label{H(E)_in_A_thm}
Let $\mathrm{prol}\bigl(A(\mathfrak g)\bigr)\subset S(V^*)\otimes V$ be the standard (Spencer-Sternberg) prolongation of $A$. 

Then
\begin{equation}
\label{H(E)prol(A)inclusion}
H(\g)^0 \subset \prol\bigl(A(\g)\bigr).   
\end{equation}
In particular, if $\mathrm{prol}(A)$ is finite-dimensional, then $\g$ is also finite-dimensional.
\end{thm}
\begin{proof}
{Set $A(\g)_0:=A(\g)$ and denote the $i$-th Spencer-Sternberg prolongation of $A(\g)$ by $A(\g)_i$.
Let us prove by induction that for every $i\geq 0$ we have the following splitting
\begin{equation}
\label{he_i}
H(\g)_i= \bigoplus_{k=0}^i H(\g)_i\cap A(\g)_k,
\end{equation}
which implies that
\begin{equation}
\label{he_i_inclusion}
H(\g)_i\subset \bigoplus_{k=0}^i A(\g)_k,
\end{equation}
and therefore \eqref{H(E)prol(A)inclusion}.

To begin the proof by induction of \eqref{he_i}, first,  if $u\in H(\g)_0$ then $u\in A(\g)$, since in this case $\bar u$ in \eqref{A_def}  can be taken to be equal to  $u$. In other words, $H(\g)_0\subset A(\g)$, i.e., \eqref{he_i} holds for $i=0$.

Now assume  that \eqref{he_i} holds for some $\bar i\geq 0$ and let $u\in H(\g)_{\bar i+1}$. Then, by definition,
\begin{equation}
\label{u_der}
[u(v_1), v_2]+[v_1, u(v_2)]=0.
\end{equation}
For any $0\leq i\leq \bar i$  and $0\leq j\leq i-1$ let $\mathrm{pr}_{i,j}: H(\g)_i\rightarrow H(\g)_i\cap A(\g)_j$ be the canonical projection with respect to the splitting \eqref{he_i}.
Decompose u into the following sum:
\begin{equation}
\label{u_decomp}
u=u_0+\dots +u_{\bar i+1},
\end{equation}
where
\begin{equation}
\label{u0}
u_0|_{_{V_{-\nu}\oplus\cdots\oplus V_{-\bar i-2}}}=
u|_{_{V_{-\nu}\oplus\cdots\oplus V_{-\bar i-2}}}, \quad  u_0|_{_{V_{-\bar i-1}\oplus\cdots\oplus V_{-1}}}=0;
\end{equation}
\begin{equation}
\label{uk}
\left\{\begin{array}{ll}
u_k|_{_{V_{-\nu}\oplus\cdots\oplus V_{-\bar i-3+k}}}=0,& u_k|_{V_j}=\mathrm{pr}_{j+\bar i+1, k-1}\circ u|_{_{V_j}}\\
~&~\\
\forall \,k \text{ and } j \text{ such that } 1\leq k\leq \bar i+1,&-\bar i-2+k\leq j\leq -1.
\end{array}\right.
\end{equation}
Let us prove that 
\begin{equation}
\label{u_k_inclusion}
u_k\in H(\g)_{\bar i+1}\cap A(\g)_k,\quad \forall \, k \text{ such that }0\leq k\leq \bar i+1,
\end{equation}
which will imply \eqref{he_i} for $i=\bar i+1$ and will complete the proof by induction

Note that \eqref{u_k_inclusion} implies that  for all $k$  such that $0\leq k\leq \bar i+1$ we have
\begin{equation}
\label{u_k_der}
[u_k(v_1), v_2]+[v_1, u_k(v_2)]=0.
\end{equation}

The proof of \eqref{u_k_inclusion} is  by induction on $k$:
For $k=0$, 
from \eqref{u0} it follows that $u_0(x)\equiv u(x)\,\,\mathrm{mod}\, \g^0(\m, \eg, \fg)$ for every $x\in V$, hence, by \eqref{A_def}, $u_0\in A(\g)_0$. Consequently, $u_0$ satisfies the relation \eqref{u_der} (with $u$ replaced by $u_0$) and being of degree $i$ with respect to $H(\g)$ it belongs to $H(\g)_{\bar i+1}$.

Now assume that for some $\bar k\geq 0 $, \eqref{u_k_inclusion} holds for all $k$ such that $0\leq k\leq \bar k$ and prove it for $k=\bar k+1$. Since from \eqref{u0} and \eqref{uk},  $u_{\bar k+1}$ is a degree $\bar k+1$ element of $\mathrm{Hom}(V,V\oplus\mathrm{prol}\bigl(A(\g)\bigr)$ with respect to the grading induced by the grading of $V\oplus\mathrm{prol}\bigl(A(\g)\bigr)$, with $V$ being of degree $-1$, to prove \eqref{u_k_inclusion}  for $k=\bar k+1$, it suffices to show that \eqref{u_k_der} is valid for $k=\bar k +1$.

Let 
\begin{equation}
\label{u^k}
u^k:=u_{k+1}+\ldots+ u_{\bar i +1}.
\end{equation}
By the induction hypothesis \eqref{u_k_der}  holds for all $0\leq k\leq k$. Since $u^{\bar k}=u-\sum_{s=0}^{\bar k} u_s$ and $u$ satisfy \eqref{u_der}, we have the following. 
\begin{equation}
\label{u^k_der}
[u^{\bar k}(v_1), v_2]+[v_1, u^{\bar k}(v_2)]=0.
\end{equation}
Consequently, since $u_{\bar k+1}=u^{\bar k}-u^{\bar k+1}$,  from \eqref{u^k_der} it follows that
\begin{equation}
\label{u_k+1_der}
[u_{\bar k+1}(v_1), v_2]+[v_1, u_{\bar k+1}(v_2)]=
-[u^{\bar k+1}(v_1), v_2]-[v_1, u^{\bar k+1}(v_2)].
\end{equation}

We now analyze the following three cases separately.

{\bf Case 1:} $v_1, v_2 \in V_{-\nu}\oplus\cdots\oplus V_{-\bar i-2+\bar k}$.
In this case, by the first relation in \eqref{u0},  $u_{\bar k+1}(v_1)=0$ and $u_{\bar k +1}(v_2)=0$. Hence,  relation \eqref{uk} holds trivially in the case $k=\bar k+1$ for all such $v_1$ and $v_2$.

{\bf Case 2:} $v_1 \in V_{-\nu}\oplus\cdots\oplus V_{-\bar i-2+\bar k}$ , $v_2 \in V_j$ for some $j$ such that $i-1+\bar k\leq j\leq -1$.}
In this case, by the first relation in \eqref{u0},  $u_{k}(v_1)=0$ for all $k$ such that $\bar k +1\leq k\leq \bar i+1$, so $u^{\bar k+1}(v_1)=0$. Therefore, substituting this into \eqref{u_k+1_der}, we get
\begin{equation}
\label{u_k+1_der_2}
[v_1, u_{\bar k+1}(v_2)]=
-[v_1, u^{\bar k+1}(v_2)].
\end{equation}
Since from \eqref{u0} and \eqref{uk}, $u_{\bar k+1}$ is a degree $k$ element of $\mathrm{Hom}(V,V\oplus\mathrm{prol}\bigl(A(\g)\bigr)$ with respect to the grading induced by the grading $V\oplus\mathrm{prol}\bigl(A(\g)\bigr)$, and the degree of $v_1$ and $v_2$ in this grading is $-1$, we conclude $[v_1, u_{\bar k+1}(v_2)]\in A(\g)_{\bar k-1}$ (here $A(\g)_{-1}:=V$). On the other hand, by the same arguments and also using \eqref{u^k} for $k=\bar k +1$, $[v_1, u^{\bar k+1}(v_2)]\in \displaystyle{\bigoplus_{k=\bar k}^{\bar i+1} A(\g)_k}$. Since $A(\g)_{\bar k-1}\cap \bigoplus_{k=\bar k}^{\bar i+1} A(\g)_k=0$, both sides of relation \eqref{u_k+1_der_2} must be zero.  In particular, $[v_1, u_{\bar k+1}(v_2)]$, which finally implies \eqref{u_k_der} for $k=\bar k+1$ for the $v_1$ and $V_2$.

{\bf Case 3: $v_1 \in V_{j_1}$ , $v_2 \in V_{j_2}$ for some $j_1$ and $j_2$ such that $i-1+\bar k\leq j_1, j_2\leq -1$}
In this case, by the same degree counting arguments as in the analysis of the previous case, in \eqref{u_k+1_der}, the left-hand side belongs to  $ A(\g)_{\bar k-1}$ and the right-hand side belongs to $\displaystyle{\bigoplus_{k=\bar k}^{\bar i+1} A(\g)_k}$. So, vanishing of the left-hand side, we get \eqref{u_k_der} for $k=\bar k+1$ for the considered $v_1$ and $V_2$.

To summarize all three cases, the \eqref{u_k_der} holds for $k=\bar k+1$ and completes the proof by induction of \eqref{u_k_inclusion}, which implies  \eqref{he_i} for $i=\bar i+1$. Consequently, we completed the proof by induction of \eqref{he_i} and the proof of the theorem.
\end{proof}

We note that due to Spencer criterion~\cite{spencer}, the space $\prol\bigl(A(\g)\bigr)$ is finite-dimensional if and only if the complexification $A(\g)^{\bbC}\subset \End(V^{\bbC})$ has no rank 1 endomorphisms. 

This result can be immediately applied to the case of generalized pseudo-product structures:
\begin{thm}\label{thm-pp-finite}
    Let $(\m;\eg,\fg)$ be the symbol of a non-degenerate generalized pseudo-product structure. Then the generalized Tanaka prolongation $\g(\m;\eg,\fg)$ is finite-dimensional.
\end{thm}
\begin{proof}
By the Spencer criterion and Theorem~\ref{H(E)_in_A_thm} it is sufficient to show that the space $A(\g)^{\bbC}$ contains no rank $1$ elements.

Let $\bar \eg$ and $\bar\fg$ be the images of $\eg$ and $\fg$ under the canonical projection from $\mathfrak m$ to $V$.  Since $\eg\oplus\fg$ generates $\mathfrak m$,  it follows that $V=\bar\eg\bigoplus \bar\fg$. 
Assume by contradiction that $A(\g)$ contains an element $\bar u$ of rank $1$. Since by constructions $\bar u(\bar \eg)\subset \bar \eg$ and  $\bar u(\bar \fg)\subset \bar \fg$, $\mathrm{rank}\, \bar u=1$ imples that either $\bar u|_{\bar \eg}=0$ or $\bar u|_{\bar \fg}=0$.

Without loss of generality, assume that  
\begin{equation}
\label{u_to_f}
\bar u|_{\bar \fg}=0.
\end{equation}
 Then there exists a nonzero $w\in\bar \eg$ that generates the image of $\bar u$. In particular, there exists $e\in \eg$ such that $w=\bar u(e)$. Using the derivation property inherited from $E$, the definition of $H(E)$, and \eqref{u_to_f},  we get
$$\forall f\in \fg:\quad  [w,f]:=[\bar u(e), f]:=\underbrace{\bar u([e, f])}_0-[e,\underbrace{\bar u(f)}_0]=0.$$
Consequently, $w$ belongs to the left kernel of the form \eqref{Leviform} and so has to be zero due to the non-degeneracy assumption. This contradicts the assumption that 
$w\neq 0$, completing the proof.
\end{proof}

\section{Generalized pseudo-product structures are of finite type}
Let $(E,F)$ be a generalized pseudo-product structure on a filtered manifold $M$ and let $(\m_p;\eg_p,\fg_p)$ be its symbol at a point $p\in M$. According to Theorem~\ref{thm-pp-finite} its generalized Tanaka prolongation $\g(\m_p;\eg_p,\fg_p)$ is finite-dimensional. 

It is not difficult to show that for any $k\ge 0$ the function
\[
f_k\colon M\to \Z,\quad p\mapsto \dim\g_k(\m_p;\eg_p,\fg_p)
\]
is upper semi-continuous. Thus, possibly restricting $(E,F)$ to an open subset in $M$ we can assume that 
\begin{itemize}
    \item there exists $k_0\ge 0$ such that $\dim \g_k(\m_p;\eg_p,\fg_p)=0$ for all $k>k_0$ and all $p\in M$;
    \item for any $0\le k\le k_0$ the dimension $\g_k(\m_p;\eg_p,\fg_p)$ is constant among all $p\in M$. 
\end{itemize}
We shall say in such case that $\g(\m_p;\eg_p,\fg_p)$ are \emph{isomorphic as graded vector spaces} for all $p\in M$. 

The main result of this section is the following
\begin{thm}
\label{Finite_Dim_symmetry_Thm}
Let $(E,F)$ be a non-degenerate generalized pseudo-product structure on a smooth filtered manifold $M$. Assume that the generalized Tanaka prolongations $\g(\m_p;\eg_p,\fg_p)$ are isomorphic as graded vector spaces for all $p\in M$.  

Then there exists a sequence of bundles
\[
M\leftarrow P_0 \leftarrow \dots \leftarrow P_{k_0}=P
\]
of dimensions $\dim P_k =\sum_{i\le k}\g_{i}(\m;\eg,\fg)$ and an absolute parallelism on $P$ naturally associated with $(E,F)$. In particular, the symmetry algebra of $(E,F)$ is finite-dimensional and does not exceed $\dim \g(\m;\eg,\fg)$.
\end{thm}
\begin{proof}
    We follow the constructions and terminology of~\cite[Sections 3--5]{hh24}, based in turn on \cite{hm24, zeltan}, with minor modifications. 
    
    Fix a graded vector space $\vg$ isomorphic to $\g(\m_p;\eg_p,\fg_p)$ for all $p\in M$ and two graded subspaces $\ve,\vf\subset \vg_{-}$ such that $\ve_i$ and $\vf_i$ are complementary and have the same dimensions as $\eg_i$, $\fg_i$ for all $i<0$.
    
    Consider the algebraic variety $\cL$ of all graded Lie algebra structures $\g$ on $\vg$ such that
    \begin{enumerate}
        \item $\g_{-}$ is a graded Lie algebra;
        \item $\ve$ and $\vf$ are subalgebras of $\g_{-}$;
        \item $\g=\g(\vg_{-},\ve,\vf)$.
    \end{enumerate}
    It is clear that $\cL$ can be viewed as an algebraic subvariety the vector space $\Hom_0(\wedge^2\vg,\vg)$ of degree presevring skew-symmetric products on $\vg$.

    Let $G$ be an algebraic subgroup of $GL(\vg)$ that preserves the grading and both the subspaces $\ve$ and $\vf$. It acts naturally on $\cL$. Let $\cT$ be the Zariski closure of the orbit space containing all orbits that correspond to the generalized Tanaka prolongations $\g(\m_p;\eg_p,\fg_p)$ of symbols $(\m_p;\eg_p,\fg_p)$ at different points $p\in M$.
    
    Using the standard results from algebraic geometry (Rosenlicht~\cite{rosen1956}, Popov--Vinberg~\cite[\S4]{popov-vinberg}) we know that the space $\cT/G$ can be equipped with the structure of an algebraic variety such that the projection $\cT\to \cT/G$ is rational. This implies that there exist a Zariski open subset $B\subset \cT/G$ and a quasi-section $\sigma \colon B\to \cT$ such that $\g(\m_p;\eg_p,\fg_p)$ is isomorphic (as graded Lie algebra) to a finite number of $\sigma(b)$, $b\in B$ for an open subset of points $p\in M$. Restricting to an open subset  in $B$ (in the induced smooth topology), we can assume that $\sigma(b)$ is unique. Again, restricting the pseudo-product structure to this subset if necessary, we assume that this holds for all points $p\in M$.

    We then proceed with the constructions of~\cite[Sections 3--5]{hh24} with the following modifications:
    \begin{enumerate}
        \item we work in the category of smooth manifolds;
        \item the model vector bundle $\vg\times B$ is equipped with two distinguished subbundles $\ve\times B$ and $\vf\times B$, such that a vector bundle homomorphism
        \[
        \Lambda \left(\wedge^2\vg\right)\times B \to \vg\times B
        \]
        turns $\vg$ for each $b\in B$ into a graded Lie algebra denoted by $\g(b)$ with two subalgebras $\eg(b)$ and $\fg(b)$, and, moreover, $\g(b)$ coincides with the generalized Tanaka prolongation $\g\left(\g(b)_{-},\eg(b),\fg(b)\right)$;
        \item as we do not assume that $\g(b)_{-}$ is fundamental, we define $\Tor_{n+1}(\vg)$ for all $n\ge 0$ as:
        \[
        \Tor_{n+1}=\Hom_{n+1}(\vg_{-}\wedge\vg_{-},\vg)\oplus \Hom(\oplus_{i=0}^{n-1}\vg_{-}\wedge \vg_{i},\vg_{n-1}).   
        \]
    \end{enumerate}

    Given the above setup, let $M$ be a smooth manifold $M$ equipped with a submersion $\beta\colon M\to B$. We define a \emph{generalized $\beta$-pseudo-product structure on $M$} as a generalized pseudo-product structure whose symbol at each point $p\in M$ is isomorphic to $(\g(b);\eg(b),\fg(b))$ for $b=\beta(p)$.

    Then we apply the same argument as in the proofs of Theorem 3.4 and its Corollary 3.6 from \cite{hh24} to prove that there is a sequence of bundles: 
    \[
        M\leftarrow P_0 \leftarrow \dots \leftarrow P_{k_0}=P
    \]
    an absolute parallelism on $P$ naturally associated with the generalized $\beta$-pseudo-product structure on $M$. 
\end{proof}

\section{Finite type distributions via abnormal extremals}
\label{abnormal_sec}
In this section, we apply the developed theory to obtain very general conditions for finite dimensionality of symmetry group of bracket-generating distributions in terms of properties of their abnormal extremals.

\subsection{Pseudo-product structure via symplectificaiton.} 
Let $D$ be a distribution on a manifold $M$.
In general, $D$  is not equipped with a natural pseudo-product structure. However, the lift of $D$ to a submanifold of the cotangent bundle $T^* M$, naturally assigned to $D$, does possess this property. We also refer to the process of passing to the cotangent bundle and working with the natural geometric structures induced there by the original distribution as \emph{symplectification} \cite{BI09, BI16, Doubrov_Day_Zelenko}. 

In more details, let $\{T^{-i}M\}_{i>0}$ be the weak derived flag of $D$, as in \eqref{week_flag}. In this case  $T^{-i}M$ is called the \emph{$i$th power} of the distribution and usually denoted by $D^i$.

To a distribution $D$ on a manifold $M$, assign a new (filtered) distribution on a submanifold of the cotangent bundle of $M$ equipped with the intrinsic generalized pseudo-product structure. 

To begin, let
$$T^*M=\{(p,q): q\in M, p\in T_q^*M\}$$
denote the cotangent bundle of $M$, $\pi: T^*M\rightarrow M$ be the canonical projection, $s$ be the tautological (Liouville) $1$-form \footnote{if $\lambda\in T^*M$, $\lambda:=(p,q): q\in M, p\in T_q^*M$, and $Y\in T_\lambda T^*M$, then  $s_\lambda (Y):=p\bigl(d \pi_\lambda(Y)\bigr)$} and  $\sigma=d\, s$ be the canonical symplectic form on $T^*M$.
The dual objects to the $j$-th power $D^j$ of $D$  on $T^*M$ is the annihilator $(D^j)^\perp$, defined as follows
$$(D^j)^\perp=\{(p,q)\in T^*M:p(v)=0\quad\forall\, v\in D^j(q)\}.$$


The following objects, which arise in optimal control theory (more specifically, when applying the Pontryagin Maximum Principle to any optimal control problem on the space of curves tangent to the distribution $D$), are central to this section:

\begin{df}
An absolutely continuous nontrivial \footnote{ nontrivial means that its image is not a point} curve $\gamma:[0, T]\mapsto D^\perp$ is called \emph{abmormal extremal} of $D$ if $\gamma'(t)\in  \ker \sigma(\gamma(t))|_{D^\perp}$ for almost every $t$. The  projection of $\gamma$  to $M$ is  called \emph{abnormal extremal trajectories} of $D$. 
\end{df}

By construction, abnormal extremals must lie in the following subset of $D^\perp$: 
\[
\widetilde W_D=\{\lambda \in D^\perp: \ker \sigma(\lambda)|_{D^\perp}\neq 0\}.
\]

Since the codimension of $D^\perp$ in  $T^*M$ is equal to $\mathrm{rank} \,D$, from the elementary properties of skew-symmetric forms it follows that
\begin{equation}
\label{tildeWDdescription}
 \widetilde W_D=\begin{cases}D^\perp& \text{if } \mathrm{rank} \,D\text { is odd,}\\ \{\lambda\in D^\perp:\Bigl (\bigwedge ^{\dim M-\mathrm{rank} \,D/2}(\sigma|_{ D^\perp})\Bigr)(\lambda)=0\} &  \text{if } \mathrm{rank} \,D\text { is even.}
\end{cases}
\end{equation}

Let $W_D$ be equal to $D^\perp$ when $\mathrm{rank}\,D$ is odd. If $\mathrm{rank}\,D$ is even, let $W_D$ be the subset of $\widetilde{W}_D$ consisting of points such that, in some neighborhood, the portion of $W_D$ within that neighborhood is a codimension-1 embedded submanifold of $D^\perp$.
For any
$\lambda\in W_D$  define
\begin{equation}
E(\lambda)=  \ker \sigma(\lambda)|_{T_\lambda  W_D}. 
\end{equation}
$E$ is called the \emph{characteristic distribution} on $W_D$. Since by construction $W_D$ is odd-dimensional, $E$ is a distribution of positive odd rank  in a neighborhood of a generic point, and it is integrable there. 
Define
\begin{equation}
\mathcal A(\lambda)=  \ker \sigma(\lambda)|_{D^\perp}\cap T_\lambda  W_D.
\end{equation}
$\mathcal A$ is called the \emph{abnormal distribution} of $D$. Obviously, horizontal curves of the abnormal distribution $\mathcal A$ are \emph{abnormal extremals}\footnote{In general, these do not constitute all abnormal extremals, as they potentially exclude those containing points in $\widetilde W_D \setminus W_D$.}   of the distribution $D$.

By constructions, 
\begin{equation}
\label{abn+_char _inclusion}    
\mathcal A\subset E,
\end{equation}
Using elementary properties of skew-symmetric forms it is easy to prove the following 

\begin{lem}
\label{AE_lemma}
Assume that $\lambda\in W_D$. Then  
\begin{enumerate}
\item $\mathcal A(\lambda)= E(\lambda)$ if and only if  $\ker \sigma(\lambda)|_{D^\perp}$\ is transversal to  $T_\lambda  W_D$ in $T_\lambda D^\perp$;
\item If $\mathcal A(\lambda)\subsetneqq E(\lambda)$, then $\dim E(\lambda)-\dim \mathcal A(\lambda)=1$.
\end{enumerate}
In particular, 
$\mathcal A= E$
if $\mathrm{rank}\, D$ is odd or $\mathrm{rank} \, \mathcal A=1$. 
\end{lem}


\begin{df}Let 
\begin{equation}
\label{MD_eq}
m_D:=\min\{\dim\, \mathcal A(\lambda): \lambda \in W_D\}
\end{equation}
and \[\mathcal R_D^{\mathrm{reg}}=\{\lambda\in  W_D: \dim \, {\mathcal A}(\lambda)=m_D\}.\] 
An abnormal extremal lying in ${\mathcal R}_D^{\mathrm{reg}}$ are called \emph{regular} and its projection to $M$ is called \emph{abnormal regular extremal trajecotriy}.
\end{df}

We consider the following filtrations on the distribution $E$
\begin{enumerate}
    \item If $\mathcal{A} = E$, then the filtration on $E$  is trivial: $E^{-1}:= E$.
    \item If $\mathcal{A} \subsetneqq E$ but $\mathcal A^2 \subsetneqq E$ (the latter inclusion is equivalent to the fact that $\mathcal A$ is involutive), then the filtration on $E$  is trivial: $E^{-1}:= E$.
    \item If $\mathcal{A} \subsetneqq E$ and \begin{equation}
\label{A^2}
\mathcal{A}^2 = E,
\end{equation}  then we set $E^{-1} := \mathcal{A}$ and $E^{-2} := E$.
\end{enumerate}
Note that in the case where $\mathcal{A} = E$, the relation \eqref{A^2} holds automatically, since $E$ is involutive, whereas for $\mathcal{A} \subsetneqq E$, this is an additional assumption.
Further, let $F$ be the vertical distribution of the bundle $W_D
\stackrel{\pi}{\rightarrow} M$ consisting of tangent spaces to the fibers. Obviously, $F$ is an integrable distribution. 
Apply the procedure of Example \ref{osculation}: let $E^i$ and $F^i$ be the decreasing filtration on $E$ and $F$ as defined in \eqref{E_filt_down} and \eqref{F_filt_down}, respectively.

\begin{thm}
\label{thm_control}
With the filtration on  $E$ according to items (1)-(3) above, if $E^0=0$ on an open set of $W_D$ and the distribution $E+F$ is bracket-generating, then the original distribution $D$ is of finite type.   
\end{thm}

\begin{proof}
For generic subset $\mathcal R_D$ of $W_D$, for any $\lambda\in \mathcal R_D$ dimensions of $E^i$ and $F^i$ are locally constant in a neighborhood $U$ of $\lambda$. By the assumption on $E^0$, $E^0=0$ in this neighborhood.
Restricting to such a neighborhood, let $\nu$ be the \emph{minimal} integer such that 
$F^{\nu+1}=
F^\nu$. By Remark \ref{Int_rem} the distribution $F^\nu$ is integrable.  Let $\widetilde U$ be the quotient space of $U$ by foliation of integral submanifolds  of $F_\nu$ in $U$,let $\Phi: U\rightarrow \widetilde U$ be the canonical projection, and let $\widetilde E:= \Phi_* E$, $\widetilde F:= \Phi_* F$, and, more generally,
\begin{equation}
\label{tilde_quotient}
\widetilde F^i=\Phi_* F^i, \widetilde E^i=\Phi_* E^i.
\end{equation}

By assumption $\widetilde E^0=0$ and by construction  
$\widetilde F^\nu=0$. Besides, the filtrations $0=\widetilde E^0\subset \widetilde E^{-1}=\widetilde E$ and $0=\widetilde F^{\nu}\subset \widetilde F^{\nu-1}\subset\cdots\subset \widetilde  F^{-1}=\widetilde F $  are the Freeman filtrations of the pair $(\widetilde E, \widetilde F)$ and the distibution $\widetilde E+\widetilde F$ is bracket-generating on $\widetilde U$. Therefore, as in Example \ref{osculation}, the pair $(\widetilde E,\widetilde F)$ defines a non-degenerate generalized pseudo-product structure on $M$ with the filtration defined as in \eqref{extended_filt} and therefore by Theorem \ref{Finite_Dim_symmetry_Thm} it has a finite-dimensional symmetry group.

Finally, because the construction of the pseudo-product structure $(\widetilde E, \widetilde F)$ from the original distribution $D$ is intrinsic, every symmetry of $D$ induces the corresponding symmetry of the pair $(\widetilde E, \widetilde F)$. Since distinct symmetries of $D$ yield distinct symmetries of $(\widetilde E, \widetilde F)$, it follows that the symmetry group of $D$ is finite-dimensional.
\end{proof}
\begin{rem} 
One can use the trivial filtration on $E$ in Theorem \ref{thm_control} also in the case where \eqref{A^2} holds. However, $E^0$ for the trivial filtration contains (and, in general, is strictly larger than) $E^0$ for the more refined filtration of item 3. Consequently, the vanishing of the former implies the vanishing of the latter, but not vice versa. Therefore, using the more refined filtration for item 3 in Theorem \ref{thm_control} yields a stronger result in this case.
\end{rem}



\begin{df}
A distribution $D$ on a connected manifold $M$ is called \emph{controllable by regular abnormal extremal trajectories} if any two generic points of $M$ can be connected by a (finite) concatenation of regular abnormal extremal trajectories.
\end{df}
As a consequence of Chow-Rashevskii theorem \cite{chow, rashevskii}, a distribution is controllable by regular abnormal extremals if $F+\mathcal A$ is bracket-generating; in the real analytic category, this condition is both necessary and sufficient. Furthermore, by\eqref{abn+_char _inclusion} bracket-genericity of $F+\mathcal A$ implies that $F+E$ is also bracket-generating. Therefore, Theorem \ref{thm_control} yields the following 
\begin{cor}
\label{cor_control}
With the filtration on  $E$ according to items (1)-(3) above,  if $E^0 = 0$ on an open subset of $W_D$ and the distribution $\mathcal{A} + F$ is bracket-generating, then the original distribution $D$ is of finite 
 type. Moreover, in the real analytic category, a distribution $D$ with $E^0 = 0$ is of finite type if it is controllable by abnormal extremal trajectories.   
\end{cor}
Finally note that if $m_D=1$, were $m_D$ is defined by \eqref{MD_eq}, then by the last sentence of Lemma \ref{AE_lemma}, $\mathcal A=E$, so $E$ is of rank $1$ at a generic point (or generic abnormal extremals of $D$  are of \emph{minimal order} in the terminology of \cite{chitour}). In this case $E^0=0$ automatically, so we have the following

\begin{cor}
\label{cor_min_rank}
 If $m_D=1$ (i.e., generic abnormal extremals are of minimal order)  and  the distribution $F+\mathcal{A}$ is bracket-generating, then the original distribution $D$ is of finite type. Moreover, in the real analytic category, a distribution $D$ with $m_D=1$ is of finite type if it is controllable by abnormal extremal trajectories.   
\end{cor}

The question of how to refine or remove the assumption $E^0=0$ in Theorem \ref{thm_control} and Corollary \ref{cor_control} remains open. This assumption is crucial in Example \ref{osculation}, as without it, the weight shift in \eqref{extended_filt} would need to affect not only the filtration of $F$ but also the filtration of $E$, and the latter cannot be done without failing to define a filtered structure.

\subsection{Discussion}

\subsubsection{Examples of distributions of finite type which are not controllable by abnormal extremals} A classical example of a vector distribution of infinite type not controllable via abnormal extremals is a Goursat distribution, i.e., rank 2 distributions with the big growth vector with constant jump of dimensions equal to $1$. In this case, all abnormal extremal trajectories are integral curves of the characteristic distribution and thus form a foliation. On the other hand, there are known examples of vector distributions of finite type, which are at the same time not controllable via abnormal extremals. For example, any Cartan prolongation of an arbitrary rank 2 distribution \cite{Doubrov_Day_Zelenko}  with weak derived flag of dimensions $(2,3,5,\dots)$ would fall in this class~\cite[Theorem 1.3]{Day_Zelenko}, \cite[subsection 7.1]{BI09}, \cite[subsection 5.1]{BI06}.  This indicates that controllability by abnormal extremals is not a necessary condition for finiteness of type.    

\subsubsection{Complexification when abnormal extemals do not exist}
In the case of even $\mathrm{rank}\,D$, the set $W_D$ (as defined in the second line of \eqref{tildeWDdescription}) may be trivial, i.e., equal to the zero section of $T^*M$ (equivalently, it may be empty when passing to the projectivized cotangent bundle), as seen in the case of fat distributions (\cite{montbook}). However, the results of this section remain applicable to such distributions via appropriate complexification, at least within the real analytic category. 

Briefly (see \cite{jmz21} for more details about this complexification), starting with a real analytic manifold $M$, we can locally consider a complex manifold $\mathbb{C}M$ (the complexification of $M$) by extending the real analytic transition maps to holomorphic functions. We then extend the real analytic distribution $D$ locally to its complex counterpart. Similarly, we consider the complex cotangent bundle $T^*\mathbb{C}M$, where the fibers are complex vector spaces. By applying the constructions from \eqref{tildeWDdescription} to this complex bundle, we define complex abnormal extremals and controllability (on $\mathbb{C}M$) by complex abnormal extremal trajectories. With these constructions, Theorem~\ref{thm_control}, Corollary~\ref{cor_control}, and Corollary~\ref{cor_min_rank} remain valid, provided all real objects are replaced by their complex analogues.

\subsubsection{Overview of previous results on distributions of maximal class}  The results of this section establish only the finiteness of type of distributions in terms of the pseudo-product structure $(E, F)$  on $W_D$, but do not provide more explicit information about the generalized Tanaka prolongations of the symbols of this pseudo-product structure. In our previous works \cite{BI06, BI09, BI08, BI16, BI25}, such information was already obtained for more restrictive classes of distributions, known as distributions of maximal class, and under the assumption that $m_D=1$. 

Let $\mathcal J$ be the lift of the distribution  $D$ to $W_D$, i.e. $J=\pi^*D$, or in more details 
\begin{equation}
\label{distrlift}
J(\lambda):=\{v\in T_\lambda W_D: d\pi_{\lambda}\in D\bigl(\pi(\lambda)\bigr)\}.
\end{equation}
Following \cite{BI16}, a distribution $D$ is said to be of \textit{maximal class} at a point $q \in M$ if, for some (and therefore for a generic) point $\lambda$ in the fiber of $W_D$ over $q$, there exists an integer $c$ such that:
\begin{equation}
\label{max_class_rel}
(\mathrm{ad}\,\mathcal{A})^c\, \mathcal J(\lambda) = \ker s|_{W_D} (\lambda).
\end{equation}

Here, $s|_{W_D}$ is the restriction of the tautological Liouville $1$-form $s$ to $W_D$, and $(\mathrm{ad}\,\mathcal{A})^c\,\mathcal J$ denotes the span of all iterative Lie brackets of $c$ (or, equivalently, up to $c$) sections of $\mathcal{A}$ with one section of $\mathcal J$.  Note that since $\mathcal{A} \subset E$ and  $E$ belongs to the  Cauchy characteristic distribution of $s|_{W_D}$ \footnote{In fact, this Cauchy characteristic distribution  of $s|_{W_D}$ is spanned by $E$ and the Euler field, the generator of homotheties in the fibers of $W_ D$.}, the left-hand side of \eqref{max_class_rel} is always contained in the right-hand side. 
In control-theoretic language, in the case of $m_D=1$,  $D$ is of maximal class at $q$ if and only if generic regular abnormal trajectory has corank $1$ \cite{AS1996, Day_Zelenko}. \footnote{Abnormal extremal trajectories are equivalently defined as critical points of the endpoint map on the space of all horizontal curves starting from a given point, and their coranks are defined as the codimension of the image of the differential of the endpoint map at them \cite{AS1996, ABB}, so that corank $1$ is the minimal possible corank for a critical point.}

Finiteness of type of distribution of maximal class is ensured under the additional assumption that we call \emph{strong maximality of class}: A distribution $D$ is said to be of \textit{strongly maximal class} at a point $q \in M$ if, for some (and therefore for a generic) point $\lambda$ in the fiber of $W_D$ over $q$, there exists an integer $c$ such that:
\begin{equation}
\label{max_class_rel_strong}
(\mathrm{ad}\,\mathcal{A})^c\, F(\lambda) = \ker s|_{W_D} (\lambda).
\end{equation}
By the same arguments as before the left-hand side of \eqref{max_class_rel_strong} is always contained in the right-hand side.

It is clear that for distributions of strongly maximal class, the sum $E + F$ is bracket-generating. Consequently, Corollary \ref{cor_min_rank} implies that the distribution $D$ of strongly maximal class with $m_D=1$ is of finite type. 
Note that the condition  $m_D=1$ automatically holds for bracket generating distributions of rank $2$ and $3$, to which \cite{BI06, BI09, BI08, BI25} were devoted, and it was also assumed in \cite{BI16}).

We have already explicitly described the maximally symmetric distributions of strongly maximal class and their infinitesimal symmetry algebras for rank $2$ distributions in \cite{BI06, BI09} and for rank $3$ distributions in \cite{BI08, BI25}. Moreover, we found rather explicit description of upper bounds for the symmetry dimensions of distributions of strongly maximal class of arbitrary rank with $m_D=1$ in \cite{BI16}. \footnote{The term of a distribution of strongly maximal class is not used in \cite{BI16}, but in the statements about finite type there we implicitly used strong maximality of class condition in terms of the form of the skew Young diagrams of the Jacobi symbol of the distribution.} Besides, as shown in \cite{Day_Zelenko}, all rank $2$ distributions that are not Goursat are essentially (i.e., up to Cartan deprolongations) of strongly maximal class \footnote{For rank 2 distributions the notion of strongly maximal class and maximal class coincide.}; thus, their finiteness of type and even the explicit description of their maximal infinitesimal symmetry algebras follow from the mentioned previous works. 

In contrast, rank $3$ distributions of non-maximal class that are still controllable by regular abnormal extremal trajectories were found recently in \cite{Doubrov_Day_Zelenko_26}, so Corollary \ref{cor_min_rank} provides a new finiteness of type result for them. Finally, in the case of rank $4$ distributions of maximal class with $m_D>1$, a more explicit description of maximal symmetry algebras will be given in \cite{Keeler_Zelenko}.

\section{Application to ODEs of mixed order}
In this section, we shall fix a natural number $n\ge 2$ and will use multi-indexes of length $n$ consisting of non-negative integers. For two such multi-indexes $\kappa=(k_1,\dots,k_n)$ and $\lambda=(l_1,\dots,l_n)$ we shall say that $\kappa\ge \lambda$ if $k_i\ge l_i$ for all $i=1,\dots,n$. We shall also perform arithmetic operations with these multi-indexes treating them as vectors. The case of $n=2$ (systems of two ODEs of mixed order) was treated in~\cite{mixed-order}.

\subsection{Trivial systems}
A trivial system of ODEs of mixed order $\kappa=(k_1,\dots,k_n)$, $n\ge 2$, is defined as a system of ordinary differential equations 
\[
(y^i)^{(k_i)}=0,\quad 1\le 1\le n
\]
on unknown functions $y^i(x)$, $i=1,\dots,n$. Let $J^\kappa=J^\kappa(\R,\R^n)$ be the jet space of mixed order. Denote by $y^i_j$, $1\le i\le n$, $1\le j\le k_i$ its jet coordinates. So this system can be interpreted as a submanifold $\E\subset J^\kappa$ given by equations $\{y^i_{k_i}=0\}$, $i=1,\dots,n$. 

The geometry of the jet space $J^\kappa$ is encoded via its standard distribution given by 
\[
\{dy_j^i-y_{j+1}^i dx=0: 1\leq i\leq n, 0\leq j\leq k_i-1\}.
\]

Define a 1-dimensional vector distribution $E$ on $\E$ as the restriction of this standard distribution to $\E$. Explicitly, in case of a trivial system of mixed order the distribution $E$ has the form:
\[
E = \left\langle \dd{x}+y^1_1\dd{y^1_0} + \dots + y^1_{k_1-1} \dd{y^1_{k_1-2}} +\quad \dots \quad +
y^n_1\dd{y^n_0} + \dots + y^n_{k_n-1} \dd{y^n_{k_n-2}} \right\rangle. 
\]

Systems of ODEs of uniform order are typically viewed under so-called point or contact transformations. The first ones are lifts from $J^0(\R,\R^n)=\R\times \R^n$, while the second ones make sense only for $n=1$ and are lifts from $J^1(\R,\R)$.

In case of jets of mixed order there are more exotic choices, essentially coming from the lifts of transformations from $J^{\kappa-\lambda}$, where $\lambda=(l_1,\dots,l_n)$ is a multi-index smaller or equal to $\kappa$ (i.e. $l_i\le k_i$ for all $i=1,\dots,n$). This motivates the following definition.
\begin{df} Given a pair of multi-indexes $\lambda\le \kappa$ of length $n$, an \emph{EDS of order $\kappa$ and type $\lambda$} is associated with a system of equations $(y^i)^{(k_i)}=0$ and is defined as a generalized pseudo-product structure on $\E=\{y^1_{k_1}=\dots=y^n_{k_n}=0\}\subset J^\kappa$ given by the above 1-dimensional vector distributions $E$ and a distribution $F$ tangent to the fibers of the projection $\E\to J^{\kappa-\lambda}$.
\end{df}

Explicitly, 
\[
F = \left\langle \dd{y^1_{k_1-l_1+1}}, \dots, \dd{y^1_{k_1-1}},\quad \dots,\quad \dd{y^n_{k_n-l_n+1}}, \dots, \dd{y^n_{k_n-1}}\right\rangle.
\]

Although the definition makes sense for arbitrary values of $\lambda\le \kappa$, we shall be mainly interested in non-trivial cases, when $\lambda\ge \mathbf{2}$. Here for any positive integer $r$ the boldface $\mathbf{r}$ will denote the multi-index $(r,r,\dots,r)$. 

There is a graphical way to encode an arbitrary EDS of type $(\kappa,\lambda)$ using a skew Young tableau (i.e., the tableau does not need to be aligned to the left). Namely, the tableau consists of $n$ rows with $k_1,\dots,k_n$ boxes respectively, the rows are initially left aligned and then the $i$-th row is shifted further left by $l_i$ boxes.

It is possible to generalize the classical Lie--Backlund theorem to the case of mixed jet spaces as follows. 
\begin{lem} Assume that $\lambda\ge \mathbf{2}$ and $\kappa-\lambda\ge \textbf{1}$. Then any (local) transformation of  $J^\kappa=J^\kappa(\R,\R^n)$, $n\ge 2$ preserving the standard distribution and the fibers of the projection $J^\kappa\to J^{\kappa-\lambda}$ also preserves fibers of the projection $J^\kappa\to J^{\kappa-\lambda-\mathbf{1}}$.
\end{lem}
\begin{proof}
	Indeed, let $C$ be the standard contact distribution and $V$ the completely integrable distribution defined as $\ker \pi_*$ for $\pi\colon J^\kappa\to J^{\kappa-\lambda}$. Then under the conditions of the lemma the distribution 
	\[
	V' = \{X\in V \mid [C,X]\subset V\}
	\]
	is tangent to fibers of the projection $J^\kappa\to J^{\kappa-\lambda-\mathbf{1}}$ and is preserved under all local transformations preserving $C$ and $V$.
\end{proof}

Thus, the skew Young tableaux corresponding to $(\kappa,\lambda)$ and $(\kappa,\lambda+\mathbf{r})$ for all positive $r$ such that $\kappa\ge \lambda+\mathbf{r}$ define EDSs which are essentially equivalent from the geometric point of view, and, for example, have the same symmetry algebras.  
Such graphical notation provides an easy way to describe the symbol of the pseudo-product structure defined by the pair $(E,F)$. Namely, introduce a negatively graded vector space $V$ of dimension equal to the number of boxes in the Young tableau as follows. Each box in the tableau corresponds to a basis element in $V$. The grading is defined in such way that it decreases by $1$ from left to right starting from $-1$ with basis elements in the same column having the same degree. Next, define an operator $X\in \End(V)$ of degree $-1$ which maps each basis element to another basis element corresponding to the box on the right, or to $0$, if there is no right neighbor. The corresponding graded nilpotent Lie algebra $\m$ is defined as $\R X\oplus V$. 

The subalgebras $\eg$ and $\fg$ that define the symbol of the corresponding pseudo-product structure are defined as follows. We have $\eg=\langle X \rangle$, and $\fg$ is a graded subspace in $V$ (viewed as an abelian subalgebra) that corresponds to all boxes in the left-most \emph{complete} column (i.e. the column that contains boxes in each row) and all boxes to the left.

Applying Theorem~\ref{thm-pp-finite} we immediately arrive at the following finiteness result.
\begin{prop}
	Assume that $\textbf{2}\le \lambda \le \kappa$. Then the generalized Tanaka prolongation $\g(\m;\eg,\fg)$ is finite-dimensional. 
\end{prop}
\begin{proof}
	By construction the subalgebra $\fg$ contains elements corresponding to the left-most box in each row of the Young tableau. This ensures that the subspace $\eg+\fg$ generates $\m$. The condition $\textbf{2}\le \lambda \le \kappa$ of the Proposition is equivalent that any two rows of the corresponding Young diagram intersect vertically by at least two boxes. This implies that the element corresponding to the right-most box in each row does not belong to $\fg$. This ensures that the pair $(\eg,\fg)$ is non-degenerate. 
\end{proof}

For example the diagram corresponding to an EDS associated with a system of order $\kappa=(4,3,2)$ and type $\lambda=(4,2,2)$ has the form:
\[
\begin{ytableau}
	*(green)-1 & *(green) -2 & *(green) -3 & -4 \\
	\none & \none & *(green) -3 & -4 & -5\\
	\none & \none & *(green) -3 & -4
\end{ytableau}
\]

It corresponds (local) transformations of $J^{(4,3,2)}$ preserving the standard distribution and the fibers of the projection $J^{(4,3,2)}\to J^{(0,1,0)}$. 
The numbers in diagram boxes specify degrees of the corresponding basis elements in $V$. And the boxes that correspond to the subalegbra $\fg$ are colored green.  

In more detail, the symbol $(\m,\eg,\fg)$ is defined as a graded nilpotent Lie algebra $\m$ with basis 
\[
\{X,\quad e_{1,i},0\le i\le 3;\quad e_{2,j}, 0\le j\le 2;\quad e_{3,k}, 0\le k\le 1\},
\]
non-zero bracket relations $[X,e_{i,j}]=e_{i,j+1}$ whenever both parts of the equality are defined, and the grading:
\[
\deg X = -1,\quad \deg e_{1,i}=-i-1,\quad \deg e_{2,j}=-j-3, \quad \deg e_{3,k}=-k-3. 
\]

The subalgebras $\eg$ and $\fg$ are given by:
\[
\eg =\langle X \rangle, \quad \fg=\langle e_{1,0}, e_{1,1}, e_{1,2}, e_{2,0}, e_{3,0}\rangle .
\]

We summarize all possible non-equivalent choices of $\lambda$ for this case of $\kappa=(4,3,2)$ in Table~\ref{t1}. In agreement with Theorem~\ref{thm-pp-finite}, all corresponding generalized Tanaka prolongations are finite-dimensional\footnote{The universal Tanaka prolongations were computed using \href{https://www.realandimaginary.com/dgcv/}{Differential Geometry with Complex Variables (dgcv)} software package.}, but have different dimension depending on $\lambda$. It is easy to see from the coordinate description of the pseudo-product structures that correspond to the trivial systems of ODEs that they are \emph{flat} models with prescribed symbol $(\m,\eg,\fg)$. So, the dimensions of the prolongations coincide in these cases with the dimensions of infinitesimal symmetry algebras. 

\begin{table}
\ytableausetup{smalltableaux}
\[
\begin{array}{|c|c|c|c|}
\hline
\lambda & \text{Skew Young tableau} & \text{Preserved projection} & \text{Dim of the prolongation} \\
\hline
& & & \\
(2,2,2) & \ydiagram{4,3,2} & \quad J^{(4,3,2)}\mapsto J^{(2,1,0)} & \quad 22 \\[10mm]
(2,3,2) & \ydiagram{1+4,3,1+2}& \quad J^{(4,3,2)}\mapsto J^{(2,0,0)} & \quad 36 \\[10mm]
(3,2,2) & \ydiagram{4,1+3,1+2} & \quad J^{(4,3,2)}\mapsto J^{(1,1,0)} & \quad 19 \\[10mm]
(3,3,2) & \ydiagram{4,3,1+2} & \quad J^{(4,3,2)}\mapsto J^{(1,0,0)} & \quad 20 \\[10mm]
(4,2,2) & \ydiagram{4,2+3,2+2} & \quad J^{(4,3,2)}\mapsto J^{(0,1,0)} & \quad 51 \\[10mm]
(4,3,2) & \ydiagram{4,1+3,2+2} & \quad J^{(4,3,2)}\mapsto J^{(0,0,0)} & \quad 29 \\[10mm]
\hline
\end{array}
\]
\caption{Dimensions of generalized Tanaka prolongations for $\kappa=(4,3,2)$}\label{t1}
\end{table}

\subsection{Non-linear systems}
More generally, a non-linear system of ODEs of mixed order $\kappa$ can be defined as a condimension $n$ submanifold $\E\subset J^\kappa$ transversal to the fibers of the projection $J^\kappa\to J^{\kappa-\mathbf{1}}$. Explicitly, it can be written as:
\begin{equation}\label{eq:nlin}
(y^i)^{(k_i)}= F^i\left(x,(y^r)^{(s)}\right),\quad 1\le i,r \le n, 1\le s < k_s. 
\end{equation} 

Restricting the standard distribution to $\E$ we get a 1-dimensional distribution $E$, whose integral curves are solutions of this system of ODEs. Further, we can define $F$ as a \emph{vertical} distribution tangent to fibers of the projection $\pi\colon J^{\kappa}\to J^{\kappa-\lambda}$. As in Example~\ref{osculation} we can always define a pseudo-product structure starting from the pair $(E,F)$. In case of a trivial system of ODEs this is exactly the flat model for the symbol $(\m;\fg,\eg)$ discussed above. However, in case of an arbitrary non-linear system the corresponding symbol may differ from $(\m;\fg,\eg)$, and we need to impose additional conditions on the right hand side of the system depending on $\kappa$ and $\lambda$.

These conditions can be also seen from the corresponding skew Young tableau. Namely, complete it to a regular (non-skew) one by adding necessary boxes on the left of each row as show, for example, below:
\[
\ytableausetup{nosmalltableaux,boxsize=2.5em}
\begin{ytableau}
	\partial_{y^1_3} & \partial_{y^1_2} & \partial_{y^1_1} &  \partial_{y^1_0} \\
	\bullet & \bullet & \partial_{y^2_2} & \partial_{y^2_1} & \partial_{y^2_0} \\
	\bullet & \bullet & \partial_{y^3_1} & \partial_{y^3_0}
\end{ytableau}
\]
Each newly added box in a row $i$ corresponds to the conditions
\[
\frac{\partial F^i}{\partial y^r_s}=0
\] 
on the right hand side of~\eqref{eq:nlin}, where operators $\partial_{y^r_s}$ correspond to all boxes in the same column that were already present in the initial Young tableau. For example, in case of the above tableau these conditions restrict~\eqref{eq:nlin} to a subclass where $F^2$ and $F^3$ do not depend on $y^1_2$ and $y^1_3$.    

The existence of a natural frame bundle and an absolute parallelism for non-linear systems follows from Theorem~\ref{Finite_Dim_symmetry_Thm}. In particular, the equivalence problem for ODEs of mixed order under conditions $\textbf{2}\le \lambda\le \kappa$ is reduced to the equivalence problem of the corresponding absolute parallelisms.

\section*{Acknowledgements}
We would like to express our gratitude to  Andrei Agrachev, Jun-Muk Hwang, Tohru Morimoto, and David Sykes for fruitful discussions on the topic of this paper. Special thanks to Tohru Morimoto for pointing us to the preprint of Isao Naruki~\cite{naruki} and to David Sykes for guiding us through his software package \href{https://www.realandimaginary.com/dgcv/}{Differential Geometry with Complex Variables (dgcv)}.

\end{document}